\documentclass{article}
\usepackage[utf8]{inputenc}
\usepackage[T1]{fontenc}
\usepackage{amsthm, amsmath, amssymb}
\usepackage{microtype}

\usepackage{Baskervaldx}
\usepackage{mathpazo}

\usepackage[style=alphabetic]{biblatex}
\DefineBibliographyStrings{english}{%
	page             = {},
	pages            = {},
}

\usepackage{enumerate}
\usepackage[top=3cm,bottom=3cm,inner=3cm,outer=3cm]{geometry}
\usepackage{tikz-cd}
\usepackage{eucal}

\newtheorem{thm}{Theorem}[section]
\newtheorem{cor}[thm]{Corollary}
\newtheorem{prop}[thm]{Proposition}
\newtheorem{lem}[thm]{Lemma}

\newtheorem*{thm*}{Theorem}
\newtheorem*{cor*}{Corollary}

\theoremstyle{definition}
\newtheorem{defn}[thm]{Definition}
\newtheorem*{defn*}{Definition}

\newtheorem{exam}[thm]{Example}
\newtheorem{rmk}[thm]{Remark}
\newtheorem{nota}[thm]{Notation}
\theoremstyle{remark}

\newcommand{\Set}{\mathrm{Set}}

\newcommand{\Ab}{\mathrm{Ab}}

\newcommand{\finpSet}{\mathrm{Fin}_\ast}

\newcommand{\pSet}{\Set_\ast}

\newcommand{\Spaces}{\mathcal{S}}

\newcommand{\Mod}{\mathrm{Mod}}
\newcommand{\sMod}{s\Mod}

\newcommand{\Fin}{\mathrm{Fin}}

\DeclareFontFamily{U}{min}{}
\DeclareFontShape{U}{min}{m}{n}{<-> udmj30}{}

\newcommand{\op}{\mathrm{op}}

\newcommand{\ang}[1]{\langle #1\rangle}

\newcommand{\CC}{\mathcal{C}}

\newcommand{\ZZ}{\mathbb{Z}}

\newcommand{\fun}{\mathbb{F}_1}
\newcommand{\frB}{\mathfrak{B}}
\renewcommand{\epsilon}{\varepsilon}

\usepackage{hyperref}
\hypersetup{
	colorlinks   = true, 
	urlcolor     = blue, 
	linkcolor    = black, 
	citecolor   = red 
}
\title{The Eilenberg-MacLane Spectrum of $\fun$}
\author{Jonathan Beardsley}

\begin{document}
	
	\maketitle
	
	\vspace{.1in}
	\begin{center}
		\textit{For Jack. May we meet again, somewhere past the Oort Cloud.}
	\end{center}
	
	\vspace{.3in}
	\begin{abstract}
		Given a very special $\Gamma$-space $X$, repeated application of Segal's delooping functor produces the constituent spaces of the associated connective $\Omega$-spectrum. In particular, by applying this construction to \textit{discrete} very special $\Gamma$-spaces (a.k.a.~Abelian groups), one recovers Eilenberg-MacLane spectra. The delooping functor is entirely formal, however, and can be applied to arbitrary $\Gamma$-spaces without any conditions. Work of Connes and Consani suggests that the ``field with one element'' can be fruitfully realized as a discrete $\Gamma$-space (whose excisive approximation is the classical sphere spectrum). This note computes Segal's deloopings of this model of $\fun$ without spectrification (i.e.~prior to excisive approximation). This results in a sequence of \textit{partially symmetric monoidal partial strict $\infty$-categories} with $n^{th}$ entry being the free partially symmetric monoidal partial strict $n$-category with one object and a non-degenerate $n$-morphism. Applying geometric realization recovers the sequence of $n$-spheres with partial symmetric monoidal structure given by the fold map. This indicates that if one allows for partially defined compositions and tensor products then the suspension spectrum $\Sigma^\infty S^0$ has a canonical (partial) infinite loop space structure.
	\end{abstract}

	\section{Introduction}
	In \cite{manin-zeta-functions, manin-numbersasfunctions}, Manin posited an ``absolute geometry'' in which arithemetic could be re-interpreted as the algebraic geometry of function fields. In this setting, a projective curve is replaced by $\mathrm{Spec}(\ZZ)$ or its completion $\overline{\mathrm{Spec}(\mathbb{Z})}$, and its field of rational functions is replaced by the rational numbers $\mathbb{Q}$. There have been many different attempts to reify Manin's vision (for a survey, see \cite{lorscheid-everyone}) of ``geometry over the field with one element,'' as it has come to be known. In this paper, I will work within the formalism of Segal $\Gamma$-sets, as laid out by Connes and Consani in \cite{connesconsani-AbsoluteAlgebraSegalGammaRings, GromovNorm, connes-consani--universalarithmetic}, which I will henceforth refer to as $\fun$-modules. The theory of $\Gamma$-sets, first developed by homotopy theorists \cite{segal,bousfieldfriedlander}, is particularly amenable to a derived and higher categorical theorization. 
	
	The question remains: why consider a derived theory of $\fun$-modules at all? 	A major goal of \cite{connesconsani-AbsoluteAlgebraSegalGammaRings} and \cite{connes-consani-Segal_Gamma_Universal}, expanded upon in \cite{owen-thesis}, is the construction of a $\Gamma$-space $\mathbf{H}(D)$, in fact a functor out of $\Gamma^{\op}$ valued in 1-coskeletal homotopy types, encoding the zeroth and first cohomology of an Arakelov divisor $D$. This enables them to prove a version of the Riemann-Roch theorem for the Arakelov compactification $\overline{\mathrm{Spec}(\mathbb{Z})}$. However, their constructions are somewhat ad hoc, analogous to early attemps at writing down low dimensional cohomology groups without the benefit of derived functors. From the standpoint of contemporary homotopy theory and derived algebraic geometry (e.g.~\cite{htt,ha}), the cohomology objects of a sheaf $\mathcal{F}$ on an Arakelov variety should be the homotopy groups of a sheaf of spectra built by delooping $\mathcal{F}$ into higher Eilenberg-MacLane stacks. Such a perspective automatically provides access to higher cohomology groups.
	
	Unfortunately, this approach breaks down when working with sheaves of general $\fun$-algebras. Indeed, on any open subset that includes the Archimedean point $\{\infty\}\in\overline{\mathrm{Spec}(\mathbb{Z})}$, Connes and Consani's structure sheaf gives a set with a commutative binary operation that is non-associative and only partially defined. If one group completes to a sheaf of Abelian groups, or attemps to produce a classical spectrum from this data, one destroys all the relevant arithmetic information. In other words, by group completing one can produce a structure sheaf of ring spectra on $\overline{\mathrm{Spec}(\ZZ)}$, but it will carry no useful information about Arakelov divisors.
	
	It is necessary to introduce a new type of derived object that remembers the number theoretic (and combinatorial) structure of the base space. I would call such a thing an \textit{$\fun$-spectrum} but the most general possible definition remains elusive. In light of this, and since the goal here is to perform a \textit{computation}, I will avoid giving a formal definition in this paper. Connective versions of such things, however, are relatively straightforward to define: they should be reduced, \textit{but not necessarily special or grouplike}, functors out of $\Fin_\ast$ into a pointed category $\CC$ which allows for infinitely many ``categorical degrees of magnitude'' without enforcing any Segal conditions whatsoever (which correspond to group completion in the categorical direction). In other words, objects that have ``room'' for higher commutative algebra structure and $n$-morphisms of all dimensions but do not put any requirements (e.g.~of being everywhere defined, or single-valued, or associative) on the tensor product and composition operations. An example of this would be $\Gamma$-objects in the $\infty$-category of pointed $\infty$-tuple simplicial spaces, $s^\infty\mathcal{S}_\ast$, though it may suffice to replace $\Delta^\infty$ with a \textit{globular} localization in the form of $\Theta_\infty=\Delta\wr\Delta\wr\cdots$. Removing connectivity entirely likely requires replacing $\Gamma$ with $s^\infty\Spaces_\ast$ itself, which would generalize Heine's work describing categorical spectra as a certain class of excisive functors out of pointed $(\infty,\infty)$-categories \cite[Theorem 6.5.3]{heine-SHofHigherCats}.
	
	In any case, the purpose of this paper is to introduce, and completely compute the structure of, what should be the monoidal unit of any reasonable ($\infty-$)category of $\fun$-spectra. I'll denote this object by $\mathfrak{B}^\infty\fun$, as it is produced by infinitely delooping our model of $\fun$. It is straightforward to see that if one applies geometric realization (group completing in the ``categorical'' direction) and excisive approximation/spectrafication (group completing in the ``algebra'' direction) to this object then one recovers the classical sphere spectrum $\mathbb{S}$. There should be an analogous result localizing it to the categorical sphere spectrum of \cite{masuda-thesis}, but a rigorous proof of this is work in progress.

	\subsubsection*{Acknowledgements}
	Thanks to Philip Hackney, Kiran Luecke, Naruki Masuda, Joe Moeller, Viktoriya Ozornova, Eric Peterson, Lorenzo Riva and Martina Rovelli
	for stimulating conversations regarding this material. This work was partially supported by the Simons Foundation, Award
	ID \#853272.  This note is dedicated to my advisor Jack Morava.

	\section{Background and Notation}
	
	
	\begin{defn}
		Let $\Delta$ denote the standard simplex category, i.e.~the category whose objects are finite totally ordered sets and whose morphisms are weakly monotonic functions. Given a category $\CC$, write $s^n\CC$ for the category of functors $(\Delta^{\op})^n\to\CC$. If $\CC$ has a terminal object $\ast$, we write $\CC_\ast$ for the undercategory $\CC^{\backslash\ast}$. In the special case that $n=1$ and $\CC=\Set$, we write $\mathcal{S}$ for the category of simplicial sets and $\mathcal{S}_\ast$ for the category of pointed simplicial sets.
	\end{defn}
	

	\begin{defn}
		Write $\Fin_\ast$ for the full subcategory of the category of all finite pointed sets spanned by the sets $\ang{n}=\{\ast,1,2,\ldots,n\}$. Given a pointed category $\CC$, the category of pointed functors $\Fin_\ast\to\CC$ will often be called the category of $\Gamma$-objects of $\CC$ (based on Segal's notation $\Fin_\ast\simeq \Gamma^{\op}$).
	\end{defn}
	
	Thinking of $\Gamma$-sets as a basis for absolute algebra was introduced by Connes and Consani in \cite{connesconsani-AbsoluteAlgebraSegalGammaRings}. We build on their foundations here.
	
	\begin{defn}
		Write $\Mod_{\fun}$ for the category of pointed functors from the category of finite pointed sets, $\Fin_\ast$, to the category of all pointed sets $\Set_\ast$. Equip this category with the Day convolution symmetric monoidal structure with respect to the smash product in both the domain and codomain. Write $\fun$ for the monoidal unit, which is the inclusion $\Fin_\ast\hookrightarrow\Set_\ast$.
	\end{defn}
	
	\begin{rmk}
		Recall that a \textit{pointed functor} between pointed categories is one that preserves the zero object. In the case of $\Mod_{\fun}$ this means that we only consider functors $X\colon \Fin_\ast\to \Set_\ast$ with the property that $X\langle 0\rangle\cong\{\ast\}$.
	\end{rmk}

	\begin{rmk}
		The objects I call $\fun$-modules in this paper are called either $\mathfrak{s}$-modules or $\mathbb{S}$-modules in the work of Connes and Consani. This is not entirely inconsistent with existing usage (e.g.~\cite{localKtheory}) but is misleading. In the \textit{stable model structure} on $\Gamma$-spaces, $\fun$ is \textit{weakly equivalent} to the usual sphere spectrum, but very clearly not isomorphic. In other words, there is a localization in $\Gamma\Spaces_\ast$, sometimes called excisive approximation, $\fun\to QS^0=\Omega^\infty\Sigma^\infty S^0$. As mentioned in the introduction, passing to such a model structure, or working with excisive functors, ensures that all arithmetic and combinatorial information is lost. By writing $\fun$-module and $\Mod_{\fun}$ we clearly communicate that our intent in working with such objects is to make connections with existing work on the ``field with one element,'' e.g.~\cite{tits-F1}, \cite{manin-numbersasfunctions} and \cite{kapsmirbraids}. Therefore, we do not localize to classical spectra.
	\end{rmk}
	
	%
	
	
	Recall that the category of commutative monoids is equivalent to the category of discrete $\Gamma$-spaces that are ``special,'' in Segal's terminology. This equivalence manifests as the well-known Eilenberg-MacLane embedding. As a result, Abelian groups (and, more generally, commutative monoids) as special examples of $\fun$-modules.
	
	\begin{exam}\label{example:EilenbergMacLaneSpectra}
		Every Abelian group gives an example of an $\fun$-module. Write $H\colon \Ab\to\Mod_{\fun}$ for the standard fully faithful Eilenberg-MacLane embedding that takes an Abelian group $A$ to the functor $\langle n\rangle\mapsto A^n$ and uses the addition of $A$ and its unit element to define $HA(\phi)$ for each $\phi\in\Fin_\ast$ (see \cite[Section 0]{segal} or \cite[Example 1.2.1]{localKtheory}).
	\end{exam}
	
	\begin{exam}
		Let $\mathbb{G}$ denote the category of axiomatic projective geometries in the sense of \cite{faure-frolicher-modernPG}. In other words, pairs $(G,\ell)$ where $G$ is a set of points and $\ell$ is a ternary collinearity relation satisfying certain axioms. Then from \cite{beardsley-nakamura--projective} there is a fully faithful embedding $\mathbb{G}\hookrightarrow\Mod_{\fun}$. By \cite{beardsley-dynkin} the image of the $n$-point trivial geometry under this embedding is the $\fun$-module which to $\langle n\rangle$ associates the set of $\lambda$-algebra structures on $\langle n\rangle$ (in the sense of \cite[Definition 1.1]{dynkin-markovprocesses}).
	\end{exam}

	\begin{defn}
		Write  $\mathtt{s}\colon \Delta^{\op}\to \Fin_\ast$ for the corestriction to $\Fin_\ast$ of the (finite, canonically pointed) simplicial set $\Delta^1/\partial\Delta^1\colon \Delta^{\op}\to\Set_\ast$. Then the delooping functor $\mathfrak{B}\colon \Mod_{\fun}\to \sMod_{\fun}$ is defined as the pullback along the following composite functor:\[
		\Delta^{\op}\times\Fin_\ast\xrightarrow{\mathtt{s}\times \mathrm{id}}\Fin_\ast\times\Fin_\ast\xrightarrow{\wedge}\Fin_\ast
		\]
		Note that the delooping functor has \textit{simplicial} $\fun$-modules as its codomain. As a result, it can be iterated, although the simplicial dimension is increased at each step. By abuse of notation, we will write $\mathfrak{B}^n\colon s^k\Mod_{\fun}\to s^{k+n}\Mod_{\fun}$ for this iterated functor for any $k$ (including $k=0$).
	\end{defn}
	
	\begin{rmk}
		Note that evaluation at $[1]\in\Delta^{\op}$  (in the ``new'' simplicial coordinate) gives a right inverse to the delooping $\mathfrak{B}\colon s^n\Mod_{\fun}\to s^{n+1}\Mod_{\fun}$. This seems to be the shadow of a kind of ``loops'' functor that can be applied to these sorts of iterated deloopings. Hence their $0^{th}$ level really can be thought of as an ``infinite loop space.''
	\end{rmk}
	
	\begin{rmk}
		We reiterate that it is essential to avoid the geometric realization functor that Segal uses to reduce from bisimplicial sets to simplicial sets, thereby obtaining a delooping functor $\Gamma\Spaces_\ast\to\Gamma\Spaces_\ast$. Applying this functor will forcibly invert all $1$-morphisms, ensuring that not even the data of commutative monoids is retained.
	\end{rmk}
	
	
	
	
	\begin{defn}
		Given a $\Gamma$-object $X$ in a category $\CC$, we will refer to the simplicial object $\mathtt{s}^\ast X$, equivalently $\mathfrak{B}X\langle 1\rangle$, as \textit{the bar construction of~ $X$}. In the case that $X=HA$ for $A$ an Abelian group, this is equal to the classical bar construction.
	\end{defn}
	
	\begin{defn}\label{defn:foldbar}
		Given a category $\CC$ with finite coproducts and a terminal object $\ast$, every pointed object is uniquely an augmented commutative monoid with respect to wedge sum (the coproduct in $\CC_\ast$). For any object $X$, the unit morphism is the pointing $\ast\to X$ and the monoid structure map $X\vee X\to X$ is the fold map. This extends to a coproduct preserving functor $X^\vee\colon\Fin_\ast\to\CC$ in the usual way, and we refer to $\mathtt{s}^\ast X^{\vee}$ as the \textit{folding bar construction of $X$}.
	\end{defn}
	
	\begin{rmk}\label{rmk:foldbar as Kan extension}
		For the reader who desires a formal construction of $\mathtt{s}^\ast X^\vee$, we note that it may be constructed as a left Kan extension in nice cases. Writing $\Fin_\ast^{\leq 1}$ for the full subcategory of $\Fin_\ast$ spanned by the objects $\langle 0\rangle$ and $\langle 1\rangle$, any pointed object $\ast\to X$ in $\CC$ defines a functor $\Fin_\ast^{\leq 1}\to \CC$. Taking the left Kan extension along $\Fin_\ast^{\leq 1}\hookrightarrow\Fin_\ast$ gives $X^\vee$ and pulling back along $\mathtt{s}$ therefore gives the folding bar construction of $X$. We could also take the left Kan extension along the inclusion $\Fin_\ast^{\leq 1}\simeq\Delta^{\op}_{\leq 1}\hookrightarrow\Delta^\op$ to obtain $\mathtt{s}^\ast X^\vee$ directly.
	\end{rmk}
	
	
	
	
	\begin{nota}
		Write $\square^n\in s^n\Set$ for the $n$-cube as an $n$-fold simplicial set. More specifically, let $\square^n\in s^n\Set$ denote the free $n$-fold simplicial set on a single non-degenerate element in multidegree $[1,1,\ldots,1]$. Write $\mathfrak{S}^n$ for the quotient $\square^n/\partial\square^n$ of $\square^n$ by all of its faces. Note that $\mathfrak{S}^n$ is canonically pointed and therefore an object of $s^n\Set_\ast$.
	\end{nota}
	
	\begin{rmk}\label{rmk:reps of sn F1 mods}
		An object of $s^n\Mod_{\fun}$ can be equivalently thought of as an $n$-fold pointed simplicial object of $\Mod_{\fun}$, a $\Gamma$-object in $n$-fold simplicial sets, or a functor from $(\Delta^{\op})^n\times\Fin_\ast$ to $\Set_\ast$. We will use these three points of view interchangably, relying on the reader to interpret, from context, how to think about the object in question. The ``default,'' however, will be to think of these objects in the second sense above: $\Gamma$-objects in $n$-fold simplicial sets (i.e.~things that ``want'' to be symmetric monoidal $n$-categories). This interpretation is closest to Segal's theory of $\Gamma$-spaces \cite{segal}.
	\end{rmk}
	
	\section{The Geometry of $\frB^n\fun$}
	
	One checks immediately that the ``bar construction'' of $\fun$ is precisely the simplicial circle.
	
	\begin{lem}\label{lem:BF1 is S1}
		The simplicial set $\mathtt{s}^\ast\fun=\mathfrak{B}\fun\langle 1\rangle$ is equal to the simplicial circle $\mathfrak{S}^1=\Delta^1/\partial\Delta^1$.
	\end{lem}
	
	\begin{proof}
		This follows from the definitions of $\mathtt{s}$ and $\fun$.
	\end{proof}
	
	\begin{rmk}\label{rmk:BF1 as bar construction}
		There is another way to interpret this result. From \cite[Proposition 4.19]{beardsley-nakamura--projective} we have that $\mathtt{s}^\ast\fun$ is isomorphic to Segal's bar construction for the partial monoid $\mathbb{F}$ with two elements $\{0,1\}$ in which $1+1$ is undefined (see \cite{segal-config_spaces}). Moreover, Segal remarks therein that the bar construction of a \textit{free} partial monoid, i.e.~one in which the only possible binary additions are the ones involving the identity element, agrees with the reduced suspension. In other words, $\mathfrak{B}\fun\langle 1\rangle$ is the bar construction of $\mathbb{F}$ and therefore isomorphic to $\Sigma S^0$.
	\end{rmk}
	
	\begin{defn}
		Write $\Gamma(n)\colon \Fin_\ast\to\Set_\ast$ for the functor corepresented by the object $\langle n\rangle\in\Fin_\ast$.
	\end{defn}
	
	\begin{lem}
		There is a canonical isomorphism between $\Gamma(n)$ and the $n$-fold product $\fun\times\fun\times\cdots\times\fun$.
	\end{lem}
	
	\begin{proof}
		This follows (after carefully taking opposite categories) from the following facts: the Yoneda embedding preserves limits; $\fun$ is canonically isomorphic to $\Gamma(1)$; and $\langle n\rangle$ is canonically isomorphic to the $n$-fold coproduct $\langle 1\rangle\vee\cdots\vee\langle 1\rangle$ in $\Fin_\ast$.
	\end{proof}
	
	Because $\mathfrak{B}$ is a right adjoint, it preserves products. Moreover, products of $\fun$-modules are computed objectwise. Therefore understanding deloopings of the corepresentables $\Gamma(n)$ reduces to understanding the deloopings of $\fun$. In particular, the following corollary to the above lemmas is immediate.
	
	\begin{cor}\label{cor:barconstructionofreps}
		The simplicial set~$\mathtt{s}^\ast\Gamma(n)=\mathtt{s}^\ast \fun^n=\frB\fun^n\ang{1}$ is isomorphic to the simplicial $n$-torus $\prod_n\Delta^1/\partial\Delta^1$.
	\end{cor}
	
	In light of Remark \ref{rmk:reps of sn F1 mods}, we can consider $\mathfrak{B}\fun$ as a functor $ \Fin_\ast\to\Spaces_\ast$. Taking the comparison to $\Gamma$-spaces seriously, we might be inclined to think of this as a simplicial set (either as a space or a category), $\mathfrak{B}\fun\langle 1\rangle$, with \textit{some kind of} commutative monoid structure. First, note that from Lemma \ref{lem:BF1 is S1}, it's clear that $\mathfrak{B}\fun\langle 1\rangle$ is neither a Kan complex nor a quasicategory. There is no simplex describing how to compose the unique non-degenerate 1-simplex with itself, which is consistent with Remark \ref{rmk:BF1 as bar construction}. We can think of it as the nerve of the \textit{partial category} with one object and one non-identity morphism whose composition with itself is undefined.
	
	In the ``algebraic'' direction, the functor $\mathfrak{B}\fun\colon \Fin_\ast\to\Spaces_\ast$ is not \textit{special} in the sense of \cite{segal}, or in modern parlance, does not \textit{satisfy the Segal condition}. It does not define a commutative monoid object in $\Spaces_\ast$, even up to homotopy. However, its abortive commutative monoid structure does not fail \textit{that badly}. It still has a tensor product, it's just not defined on every simplex. These impressions are made precise by the following computations. First, we determine the underlying $n$-fold simplicial set of $\mathfrak{B}^n\fun$.
	
	\begin{prop}\label{prop:wedgeproductdelooping}
		The  $n$-fold simplicial set $\frB^n\fun\ang{k}$ is isomorphic to the $k$-fold wedge product $\bigvee_k \frB^n\fun\ang{1}$.
	\end{prop}
	
	\begin{proof}
		We prove the proposition by induction. For $n=0$ it is clearly true, since $\fun\ang{k}=\ang{k}\cong\vee_k\ang{1}$. Indeed, since $\ang{k}=\vee_{k}\ang{1}$, it suffices to show that $\frB^n\fun\colon \Fin_\ast\to s^n\Set_\ast$ preserves finite coproducts in $\finpSet$. By currying, this is equivalent to asking whether or not the associated functor $(\Delta^{\op})^n\times\finpSet\to\pSet$ preserves finite coproducts in the right-hand coordinate. Suppose that this is the case for $\frB^{n-1}\fun\colon (\Delta^{\op})^{n-1}\times\finpSet\to\pSet$. We then construct $\frB \frB^{n-1}\fun=\frB^n\fun$ as the following composite:
		\[(\Delta^{\op})^n\times\Delta^{\op}\times\finpSet\xrightarrow{id\times\mathtt{s}\times id}(\Delta^{\op})^n\times\finpSet\times\finpSet\xrightarrow{id\times\wedge}(\Delta^{\op})^n\times\finpSet\xrightarrow{\frB^{n-1}\fun}\pSet\]
		However, every functor applied to the right hand copy of $\finpSet$ in the above composite preserves coproducts in it, since $\finpSet$ is closed with respect to the smash product.
	\end{proof}
	
	\begin{lem}
		If $X\in s^n\Mod_{\fun}$, thought of as a functor $(\Delta^{\op})^n\times\finpSet\to\Set_\ast$, and $X$ is the one element set $\{\ast\}$ for every multi-degree $([k_1,\ldots,k_n],\ang{k_{n+1}})$ with $k_i=0$ for at least one $i$, then $\frB X\colon (\Delta^{\op})^{n+1}\times\finpSet\to\pSet$ is the one point set for every multidegree $([j_1,\ldots,j_n,j_{n+1}],\ang{j_{n+2}})$ whenever $j_i=0$ for at least one $i$.
	\end{lem}

	\begin{proof}
		Considering the composite from the preceding proof and noticing that $\ang{k}\wedge\ang{0}=\ang{0}$ for all $k$ completes the proof.
	\end{proof}
	
	\begin{prop}\label{prop: structure of BnF1}
		The $n$-fold simplicial set $\frB^n\fun\ang{1}$ is a singleton set in every multidegree $[k_1,\ldots,k_n]$ with $k_i=0$ for at least one $i$, has exactly two elements, one of which is non-degenerate, in multidegree $[1,1,\ldots,1]$, and has all other multisimplices degenerate.
	\end{prop}
	
	\begin{proof}
		The first claim follows by induction from the preceding proposition. The second claim follows from induction and noticing that $\frB^n\fun\ang{1}=\mathtt{s}^\ast \frB^{n-1}\fun$. The last claim follows from Proposition \ref{prop:wedgeproductdelooping} and induction.
	\end{proof}
	
	\begin{cor}
		The $n$-fold simplicial set $\frB^n\fun\langle 1\rangle$, i.e.~the bar construction $\mathtt{s}^\ast\frB^{n-1}\fun$, is isomorphic to the $n$-fold simplicial $n$-sphere $\mathfrak{S}^n$.
	\end{cor}
	
	
	\begin{thm}\label{thm:realization of BnF1 is nsphere}
		The geometric realization of the $n$-fold simplicial set $\mathfrak{B}^n\fun\ang{1}$ is homeomorphic to $S^n=D^n/\partial D^n$.
	\end{thm}
	
	\begin{proof}
		The geometric realization of the free $n$-cube is homeomorphic to the space $[0,1]^n$. The result then follows from the preceding results and the fact that geometric realization preserves colimits (and all quotients here are strict).
	\end{proof}
	
	In other words, we have computed the underlying $n$-fold simplicial sets of all deloopings $\mathfrak{B}^n\fun$. Explicitly considering the multisimplices of $\mathfrak{S}^n$, we see that this may be interpreted as the free partial strict $n$-category with one object and one non-degenerate $n$-morphism. Note that this is $n$-category is only \textit{partial} because there are no higher multisimplices describing composition (except with degenerate morphisms). This is consistent with the fact that the $n$-fold delooping of a commutative monoid is the one object $n$-category whose only non-identity morphisms are the monoid itself acting as $n$-dimensional endomorphisms of the unique object.
	
	\begin{rmk}\label{rmk:bigrmk}
		In light of Remark \ref{rmk:foldbar as Kan extension}, there is a way to understand Theorem \ref{thm:realization of BnF1 is nsphere} as a universal construction. For simplicity of exposition, let us assume that we've taken coherent nerves throughout so that $\Spaces_\ast$ now denotes the $\infty$-category of pointed anima. As such, we can write $X\in\Spaces_\ast$ as a functor $\Fin_\ast^{\leq 1}\to \Spaces_\ast$ and, again, left Kan extend along the inclusion $\Fin_\ast^{\leq 1}\hookrightarrow\Fin_\ast$ to obtain the functor $X^\vee\colon \Fin_\ast\to\Spaces_\ast$ which takes $\langle n\rangle$ to $\vee_{n}X$ and uses the unit, projection, and fold maps for the morphisms in $\Fin_\ast$. For $X=S^0$ this is precisely $\fun$.
		
		Now suppose we wish to obtain a spectrum object from $X$, specifically its suspension spectrum $\Sigma^\infty X$. We can once more left Kan extend, but now along the inclusion $\Fin_\ast\hookrightarrow\Spaces_\ast^{fin}$ of finite pointed sets into finite pointed anima. Note that the resulting functor is not excisive (in the sense of \cite[Definition 1.4.2.1]{ha}) and so does not define a spectrum object (in the sense of \cite[Definition 1.4.2.8]{ha}). When applied to the $n$-spheres $S^n$, this functor can be computed to give $\Sigma^n X$. The algebraic structure on each $\Sigma^nX$ is still the fold map $\Sigma^nX\vee\Sigma^nX\to\Sigma^nX$. Note that by forcing our functor to land in anima rather than a multisimplicial extension, we have already group completed in the categorical direction and so have lost the directedness and partialness of Proposition \ref{prop: structure of BnF1}.
		
		Group completing in the \textit{algebraic} direction corresponds to applying the excisive approximation, or spectrification, functor $\Omega^\infty\Sigma^\infty$ of \cite[Example 6.1.1.23]{ha}. This gives the excisive functor, i.e.~homology theory, $\Spaces_\ast^{fin}\to\Spaces_\ast$ which takes $S^n$ to $\Omega^{\infty-n}\Sigma^\infty X\simeq B^nQX$. The morphisms in $\Fin_\ast$ encode the infinite loop space structure of $\Omega^\infty\Sigma^\infty X$. Finally, we have obtained an honest spectrum, i.e. a very special $\Gamma$-space; specifically, this is the $\Omega$-spectrum corresponding to the suspension spectrum $\Sigma^\infty X$. Applying this sequence of left adjoints to to $S^0$ in particular ultimately gives $\mathbb{S}$.
	\end{rmk}

	
	\section{The Algebraic Structure of $\frB^n\fun$}
	
	Recall that the $n$-fold delooping of a commutative monoid $M$ is not just an $n$-category but an $n$-category with a canonical symmetric monoidal structure. The symmetric monoidal structure, however, is ``uninteresting'' in the sense that it is identical to the addition law of the monoid, which is in turn identical to the composition law for $n$-morphisms. Symbolically, given two $n$-morphisms $f$ and $g$ in $B^nM$, we have $f\circ g=f+g=f\otimes g$. We now see that the deloopings $\mathfrak{B}^n\fun$ encode the analogous algebraic structure on $\mathfrak{B}^n\fun\langle 1\rangle$ though, again, it is only partially defined.
	
	
	
	\begin{defn}
		Let $A$ be an Abelian group. Then we write $\mathfrak{K}(A,n)$ for the $n$-fold simplicial $\Gamma$-set $\mathfrak{B}^nHA$. Equivalently, in the language of Remark \ref{rmk:bigrmk}, this is the value of the spectrum $HA\colon\Spaces_\ast^{fin}\to\Spaces_\ast$ evaluated at $S^n$.
	\end{defn}
	
	Note that $\mathfrak{K}(A,1)$ is precisely the bar construction $BA$ with its usual symmetric monoidal structure (as a functor out of $\Fin_\ast$). More generally, $\mathfrak{K}(A,n)$ is the ($n$-simplicial nerve of the) strict $n$-groupoid whose only non-identity morphisms are the $n$-morphisms $A$.
	
	\begin{prop}
		For all $n$ there is a monomorphism of $n$-fold simplicial $\Gamma$-sets $\iota_n\colon \mathfrak{B}^n\fun\hookrightarrow \mathfrak{K}(\ZZ/2,n)$.
	\end{prop}
	
	\begin{proof}
		Note that if $\mathbb{F}$ is the partial commutative monoid with elements $\{0,1\}$, having $0$ as additive identity and $1+1$ undefined, then we may write $\fun\cong \widehat{H}\fun$, where $\widehat{H}$ is the generalized Eilenberg-MacLane functor of \cite{beardsley-nakamura--projective}. There is evidently a monomorphism of commutative partial monoids $\mathbb{F}\hookrightarrow\mathbb{Z}/2$. Since $\widehat{H}$ is a right adjoint \cite[Theorem 3.13]{beardsley-nakamura--projective}, and right adjoints (e.g.~$\widehat{H}$ and $\mathfrak{B}$) preserve monomorphisms, the result follows.
	\end{proof}
	
	\begin{rmk}
		
		For $n=0$, this is the level-wise inclusion of $\Gamma$-sets $\fun\hookrightarrow H\mathbb{Z}/2$. In dimension $d$, the former is precisely $\ang{d}$, which has $d+1$ elements, and the latter is $(\mathbb{Z}/2)^d\cong\ang{1}^d$, which has $2^d$ elements. The image of this inclusion in dimension $d$ is the $d$ ``axes'' of $(\mathbb{Z}/2)^d$. Alternatively, it is $\vee_dS^0\hookrightarrow\prod_d\mathbb{Z}/2$.
		
		For $n=1$, the geometric realization of $\iota_n$ is necessarily homotopic to the inclusion $S^1\simeq\mathbb{R}P^1\hookrightarrow \mathbb{R}P^\infty$. Moreover, this inclusion is a map of topological commutative partial monoids, where $S^1$ is equipped with the fold-map partial monoid structure $S^1\times S^1\supset S^1\vee S^1\to S^1$. It is, of course, not a morphism of infinite loop spaces, as $\frB\fun$ does not satisfy the Segal condition as a $\Gamma$-space.
		
		In higher dimensions, $\iota_n$ realizes to the (homotopically) unique map $S^n\hookrightarrow K(\mathbb{Z}/2,n)$. It is still a map of partial topological commutative monoids.
	\end{rmk}
	
	By abuse of notation, we will use $d_i^n\colon \ang{n}\to\ang{n-1}$ to denote the image of the $i^{th}$ face map under the inclusion $\Delta^{\op}\hookrightarrow\Fin_\ast$. For $i=0$ (resp.~$i=n$) this is the function which sends $1$ to $0$ (resp.~$n$ to $0$) and shifts all other indices down by one (resp.~is the identity on all other indices). For $0<i<n$, this is the order preserving function which conflates the elements $i$ and $i+1$, in $\ang{n}$, both to $i\in\ang{n-1}$ and shifts indices as necessary. The statement of the following proposition requires recalling from Proposition \ref{prop:wedgeproductdelooping} that $\frB^n\fun\ang{k}\cong\vee_{k}\frB^n\fun\ang{1}$.
	
	\begin{prop}
		The face maps $\frB^n\fun(d_i^k)\colon \frB^n\fun\ang{k}\to\frB^n\fun\ang{k-1}$ are given by projection onto outer summands when $i=0,k$ and by the fold map on the $i^{th}$ and $i+1^{st}$ summands when $0<i<k$.
	\end{prop}
	
	\begin{proof}
		We again prove the statement by induction. It is clearly true for $n=0$, purely by considering the functor $\mathtt{s}\colon\Delta^{\op}\to\Fin_\ast$ (see Tables 1 and 2 of \cite[Appendix A]{beardsley-nakamura--projective} for some visual intuition regarding this functor). For the general case, note that by Proposition \ref{prop: structure of BnF1}, it suffices to check the statement on the non-degenerate simplices in multidegree $[1,1,\ldots,1]$. Suppose the result holds for $n$ and consider the composite functor defining $\frB^{n+1}\fun$:
		\[\Fin_\ast\times\Delta^{\op}\xrightarrow{\mathrm{id}\times\mathtt{s}}\Fin_\ast\times\Fin_\ast\xrightarrow{\wedge}\Fin_\ast\xrightarrow{\frB^n\fun}s^n\Set_\ast\]
		Since we only care about the simplices in multidegree $[1,\ldots,1]$ we may restrict to $[1]$ in the $\Delta^{\op}$ coordinate in the above composite. In this case, because smashing with $\ang{1}$ is the identity, we simply recover the corresponding face map from $\frB^n\fun$, which is the desired function $\{0,1,...,k\}\to\{0,1,\ldots,k-1\}$.
	\end{proof}
	
	A similarly straightforward inductive analysis indicates that the remaining morphisms of $\Fin_\ast$ behave as expected under $\frB^n\fun$. This leads to the following.
	
	\begin{thm}
		As a $\Gamma$-object in $s^{n}\Set_\ast$, the functor $\frB^n\fun$ is isomorphic to $(\mathfrak{S}^n)^{\vee}$ in the sense of Definition \ref{defn:foldbar} .
	\end{thm}
	
	\begin{rmk}
		In other words, $\frB^n\fun$ is precisely the $\Gamma$-object describing the free partial commutative monoid structure on the partial strict $n$-category with a unique non-identity $n$-cell. It is, in a certain sense, the most degenerate possible object with ``$n$-categorical structure'' and ``symmetric monoidal structure.'' It has an identity morphism and a single non-trivial morphism, and there is ``space'' for things to be tensored and composed, but (except for the identity) no assumptions are made on \textit{how} precisely to do so. Morphisms from this object to a symmetric monoidal $n$-category simply pick out an $n$-endomorphism of the identity $(n-1)$-endomorphism of the tensor unit. If that symmetric monoidal $n$-category is in fact an infinite loop space, this recovers its $n^{th}$ homotopy group.
	\end{rmk}

		\printbibliography
		
	\end{document}